\documentclass[12pt]{amsart}
\pdfoutput=1
\usepackage{graphicx}
\usepackage{amssymb}
\usepackage{amsmath}
\usepackage{amsthm}
\usepackage{amscd}
\usepackage{epsfig}

\newtheorem{theorem}{Theorem}[section]

\newtheorem{lemma}[theorem]{Lemma}

\newtheorem{proposition}[theorem]{Proposition}

\newtheorem{question}[theorem]{Question}

\theoremstyle{definition}
\newtheorem{definition}[theorem]{Definition}
\numberwithin{equation}{section}

\newcommand{\Z}{\mathbb Z}
\newcommand{\N}{\mathbb N}

\begin{document}

\title{Bowen entropy for actions of amenable groups}

\author [Dongmei Zheng and Ercai Chen]{Dongmei Zheng and Ercai Chen}

\address{School of Mathematical Sciences and Institute of Mathematics, Nanjing Normal University, Nanjing 210023, Jiangsu, P.R.China
 \& Department of Applied Mathematics, College of Science, Nanjing Tech University, Nanjing 211816, Jiangsu, P.R. China} \email{dongmzheng@163.com}

\address{School of Mathematical Sciences and Institute of Mathematics, Nanjing Normal University, Nanjing 210023, Jiangsu, P.R.China} \email{ecchen@njnu.edu.cn}

\thanks{}


\begin{abstract}
Bowen introduced a definition of topological entropy of subset inspired by Hausdorff dimension in 1973 \cite{B}. In this paper we consider the Bowen's entropy for amenable group action dynamical systems and show that under the tempered condition, the Bowen entropy of the whole compact space for a given F{\o}lner sequence equals to
the topological entropy. For the proof of this result, we establish a variational principle related to the Bowen entropy and the Brin-Katok's local entropy formula for dynamical systems with amenable group actions.
\end{abstract}

\maketitle

\section{Introduction}
Let $(X,G)$ be a $G-$action topological dynamical system, where $X$ is a compact metric space with metric $d$ and $G$ a topological group.
In this paper, we assume $G$ is a discrete countable amenable group. Recall that a group $G$ is {\it amenable} if it admits a left invariant mean(a state on $\ell^{\infty}(G)$ which is invariant under left translation by $G$). This is equivalent to the existence of a sequence of
finite subsets $\{F_n\}$ of $G$ which are asymptotically invariant, i.e.,

$$\lim_{n\rightarrow+\infty}\frac{|F_n\vartriangle gF_n|}{|F_n|}=0, \text{ for all } g\in G.$$
Such sequences are called F{\o}lner sequences. For the detail of amenable group actions, one may refer to Ornstein and Weiss¡¯s pioneering paper \cite{OW}.

The topological entropy of $(X,G)$ is defined in the following way.

Let $\mathcal {U}$ be an open cover of $X$, the topological entropy of $\mathcal{U}$ is
$$h_{top}(G,\mathcal {U})=\lim_{n\rightarrow+\infty}\frac{1}{|F_n|}\log N\big(\mathcal{U}_{F_n}\big),$$
where $\mathcal{U}_{F_n}=\bigvee_{g\in F_n}g^{-1}\mathcal{U}$.
It is shown that $h_{top}(G,\mathcal {U})$ is not dependent on the choice of the F{\o}lner sequences $\{F_n\}$.
And the topological entropy of $(X,G)$ is
$$h_{top}(X,G)=\sup_{\mathcal U} h_{top}(G,\mathcal {U}),$$
where the supremum is taken over all the open covers of $X$.

Bowen \cite{B} introduced a definition of topological entropy on subsets inspired by Hausdorff dimension. For an amenable group action dynamical system
$(X,G)$, we define the Bowen topological entropy in the following way.

Let $\{F_n\}$ be a F{\o}lner sequence in $G$ and $\mathcal{U}$ be a finite open cover of $X$.
Denote $diam(\mathcal{U}) := \max\{diam(U) : U \in \mathcal{U}\}$. For $n\ge 1$ we denote by $\mathcal{W}_{F_n}(\mathcal{U})$ the
collection of families $\mathbf{U} = \{U_g\}_{g\in F_n}$ with $U_g\in \mathcal{U}$. For $\mathbf{U}\in\mathcal{W}_{F_n}(\mathcal{U})$ we call the integer
$m(\mathbf{U}) = |F_n|$ the length of $\mathbf{U}$ and define
\begin{align*}
X(\mathbf{U}) &= \bigcap_{g\in F_n} g^{-1}U_g \\
&=\{x\in X : gx \in U_g \text{ for }g\in F_n\}.
\end{align*}
For $Z\subset X$, we say that $\Lambda\subset \bigcup_{n\ge 1}\mathcal{W}_{F_n}(\mathcal{U})$ covers $Z$ if
$\bigcup_{\mathbf{U}\in\Lambda}X(\mathbf{U})\supset Z$. For $s\in\mathbf{R}$,
define
$$\mathcal{M}(Z,\mathcal{U}, N,s,\{F_n\})=\inf\limits_{\Lambda}\{\sum\limits_{\mathbf{U}\in \Lambda}\exp(-s m(\mathbf{U}))\}
$$
and the infimum is taken over all $\Lambda\subset \bigcup_{j\ge N}\mathcal{W}_{F_j}(\mathcal{U})$ that covers $Z$.
We note that $\mathcal{M}(\cdot,\mathcal{U}, N,s,\{F_n\})$ is a finite outer measure on X, and
$$ \mathcal{M}(Z,\mathcal{U}, N,s,\{F_n\})=\inf\{\mathcal{M}(C,\mathcal{U}, N,s,\{F_n\}):C \text{ is an open set that contains }Z\}.$$
$\mathcal{M}(Z,\mathcal{U}, N,s,\{F_n\})$ increases as $N$ increases. Define
$$\mathcal{M}(Z,\mathcal{U},s,\{F_n\})= \lim\limits_{N\rightarrow +\infty}\mathcal{M}(Z,\mathcal{U}, N,s,\{F_n\})$$
and
\begin{align*}
    h^B_{top}(\mathcal{U}, Z,\{F_n\}) &= \inf\{s : \mathcal{M}(Z,\mathcal{U},s,\{F_n\})= 0\} \\
    &= \sup\{s : \mathcal{M}(Z,\mathcal{U},s,\{F_n\}) = +\infty\}.
\end{align*}
Set
$$h^B_{top}(Z,\{F_n\})= \sup_{\mathcal {U}}
h^B_{top}(\mathcal{U}, Z,\{F_n\}),$$
where $\mathcal{U}$ runs over finite open covers of $Z$. We call $h^B_{top}(Z,\{F_n\})$ the Bowen topological
entropy of $(X,G)$ restricted to $Z$ or the Bowen topological entropy of $Z$(w.r.t. the F{\o}lner sequence $\{F_n\}$).

Similar to the Bowen topological entropy of subsets for $\Z-$actions(see, for example, Pesin \cite{P}), it is easy to show that
$$h^B_{top}(Z,\{F_n\})=\lim_{diam(\mathcal{U})\rightarrow 0}h^B_{top}(\mathcal{U}, Z,\{F_n\}).$$
So the Bowen topological entropy can be defined in an alternative way.

For a finite subset $F$ in $G$, we denote by
\begin{align}\label{df}
B_{F}(x,\epsilon)&=\{y\in X: d_{F}(x,y)<\epsilon\}\nonumber\\
&=\{y\in X: d(gx,gy)<\epsilon, \text{ for any }g\in F\}.
\end{align}

For $Z\subseteq X, s\ge0, N\in\mathbf{N}$, $\{F_n\}$ a F{\o}lner sequence in $G$ and $\epsilon> 0$, define
$$\mathcal{M}(Z,N,\epsilon,s,\{F_n\}) = \inf\sum_i\exp(-s|F_{n_i}|),$$
where the infimum is taken over all finite or countable families $\{B_{F_{n_i}}(x_i,\epsilon)\}$ such that
$x_i\in X, n_i\ge N$ and $\bigcup_i B_{F_{n_i}}(x_i,\epsilon)\supseteq Z$. The quantity $\mathcal{M}(Z,N,\epsilon,s,\{F_n\})$
does not decrease as $N$ increases and $\epsilon$ decreases, hence the following limits exist:
$$\mathcal{M}(Z,\epsilon,s,\{F_n\})=\lim_{N\rightarrow+\infty}\mathcal{M}(Z,N,\epsilon,s,\{F_n\}),\mathcal{M}(Z,s,\{F_n\})=\lim_{\epsilon\rightarrow0}\mathcal{M}(Z,\epsilon,s,\{F_n\}).$$
Bowen topological entropy $h^B_{top}(Z,\{F_n\})$ can be equivalently defined as the critical value
of the parameter $s$, where $\mathcal{M}(Z,s,\{F_n\})$ jumps from $+\infty$ to $0$, i.e.,
\begin{equation*}
\mathcal{M}(Z,s,\{F_n\})=\begin{cases}
                             0, s > h^B_{top}(Z,\{F_n\}),\\
                             +\infty, s < h^B_{top}(Z,\{F_n\}).
                        \end{cases}
\end{equation*}

In \cite{B} Bowen showed that $h_{top}(X,T)=h_{top}^B(X,T)$ for any compact metric dynamical system $(X,T)$. It nature to ask: Does this result also hold for
the amenable group action system $(X,G)$?

In this paper, we will prove this under certain condition on the F{\o}lner sequences. Although this is a topological problem, our proof
uses a measure-theoretic way.

A F{\o}lner sequence $\{F_n\}$ in $G$ is said to be {\it tempered} (see Shulman \cite{S}) if there exists a constant $C$ which is independent of $n$ such that
\begin{align}\label{tempered}
|\bigcup_{k<n}F_k^{-1}F_n|\le C|F_n|, \text{ for any }n.
\end{align}

In Lindenstrauss \cite{L}, \eqref{tempered} is also called {\bf Shulman Condition}.

Now we state our main theorem as follows.
\begin{theorem}[\bf{Main result}]\label{th-1-1}
Let $(X,G)$ be a compact metric $G-$action topological dynamical system and $G$ a discrete countable amenable group, then for any tempered F{\o}lner sequence $\{F_n\}$ in $G$
with the increasing condition
\begin{equation}\label{eq-1-2}
    \lim\limits_{n\rightarrow+\infty}\frac{|F_n|}{\log n}=\infty,
\end{equation}
we have
$$h_{top}^B(X,\{F_n\})=h_{top}(X,G).$$
\end{theorem}


\section{Local entropy and Brin-Katok's entropy formula}
In this section, we will prove Brin-Katok's entropy formula \cite{BK} for amenable group action dynamical systems. The statement of this formula is the following.

\begin{theorem}[Brin-Katok's entropy formula: ergodic case]\label{th-3-3}
Let $(X,G)$ be a compact metric $G-$action topological dynamical system and $G$ a discrete countable amenable group.
Let $\mu$ be a $G-$ergodic Borel probability measure on $X$ and $\{F_n\}$ a tempered F{\o}lner sequence in $G$
 with the increasing condition \eqref{eq-1-2}, then for $\mu$ almost everywhere $x\in X$,
\begin{align*}
&\lim_{\delta\rightarrow 0}\liminf_{n\rightarrow+\infty}-\frac{1}{|F_n|}\log \mu(B_{F_n}(x,\delta)) \\
=&\lim_{\delta\rightarrow 0}\limsup_{n\rightarrow+\infty}-\frac{1}{|F_n|}\log \mu(B_{F_n}(x,\delta))=h_{\mu}(X,G).
\end{align*}
\end{theorem}

Since this formula gives an alternative definition
for metric entropy(known as local entropy), we give the following definition of local entropy in amenable group action case.

\begin{definition}
Let $(X,G)$ be a compact metric $G-$action topological dynamical system and $G$ a discrete countable amenable group. Denote by $M(X)$ the collection of Borel probability measures on $X$.
For any $\mu\in M(X)$, $x\in X,n\in\mathbf{N}$, $\epsilon>0$ and $\{F_n\}$ any F{\o}lner sequence in $G$,
denote by
$$\underline h_{\mu}^{loc}(x,\epsilon,\{F_n\})=\liminf_{n\rightarrow +\infty}-\frac{1}{|F_n|}\log \mu(B_{F_n}(x,\epsilon)).$$
Then the (lower) local entropy of $\mu$ at $x$ (along $\{F_n\}$) is defined by
$$\underline h_{\mu}^{loc}(x,\{F_n\})=\lim_{\epsilon\rightarrow 0}\underline h_{\mu}^{loc}(x,\epsilon,\{F_n\})$$
and the (lower) local entropy of $\mu$ is defined by $$\underline h_{\mu}^{loc}(\{F_n\})=\int_X\underline h_{\mu}^{loc}(x,\{F_n\})d\mu.$$
Similarly, we can define the upper local entropy.
\end{definition}

For the proof of Theorem \ref{th-3-3}, we need the following classic results in ergodic theory for amenable group actions.

Let $(X,G,\mu)$ be a measure-theoretic $G-$action dynamical system where $G$ is a discrete countable amenable group and $\mu$ is a $G-$ergodic Borel probability measure on $X$.
The ergodic theorem states that,

\begin{theorem}[{E. Lindenstrauss \cite{L}}, see also {B. Weiss \cite{W}}]\label{th-3-1}\

Let $(X,G,\mu)$ be an ergodic $G-$system, $\{F_n\}$ be a tempered F{\o}lner sequence in $G$ and $f\in L^1(X,\mathcal{B},\mu)$. Then
$$\lim_{n\rightarrow+\infty}\frac{1}{|F_n|}\sum_{g\in F_n}f(gx)=\int_{X}f(x)d\mu,$$
almost everywhere and in $L^1$.
\end{theorem}

Let $\mathcal {P}$ be a finite measurable partition of $X$. For a finite subset $F$ in $G$, we denote by $\mathcal{P}_F=\bigvee_{g\in F}g^{-1}\mathcal{P}$.
Then the classical measure-theoretical entropy of $\mathcal {P}$ is defined by
$$h_{\mu}(G,\mathcal{P})=\liminf_{n\rightarrow+\infty}\frac{1}{|F_n|}H(\mathcal{P}_{F_n}),$$
where $\{F_n\}$ is any F{\o}lner sequence in $G$ and the definition is independent of the specific F{\o}lner sequence $\{F_n\}$.
The measure-theoretical entropy of the system $(X,G,\mu)$, $h_{\mu}(X,G)$, is the supremum of $h_{\mu}(G,\mathcal{P})$ over $\mathcal {P}$.

For $x\in X$, let $\mathcal {P}(x)$ denote the atom in $\mathcal{P}$ that contains $x$. Now we recall the classical Shannon-McMillan-Breiman theorem for ergodic $G-$systems.
\begin{theorem}[Shannon-McMillan-Breiman(SMB) theorem, see \cite{L,W}]\label{th-3-2} \
Let $(X,G,\mu)$ be an ergodic $G-$system. For any tempered F{\o}lner sequence $\{F_n\}$ in $G$ with $\lim\limits_{n\rightarrow+\infty}\frac{|F_n|}{\log n}=\infty$
and any finite measurable partition $\mathcal{P}$ of $X$,
$$\lim_{n\rightarrow+\infty}-\frac{1}{|F_n|}\log \mu(\mathcal{P}_{F_n}(x))=h_{\mu}(G,\mathcal{P}),$$
for $\mu$ almost everywhere $x\in X$.
\end{theorem}

Now we give the proof of Brin-Katok's entropy formula for amenable group action systems.

\begin{proof}[{\bf Proof of Theorem \ref{th-3-3}}] Let $h=h_{\mu}(X,G)$. We first prove for $\mu$-a.e. $x\in X$,
$$\lim\limits_{\delta\rightarrow 0}\limsup\limits _{n\rightarrow+\infty}-\frac{1}{|F_n|}\log \mu(B_{F_n}(x,\delta))\le h_{\mu}(X,G).$$

Let $\delta>0$ be given and let $\xi$ be a finite measurable partition of $X$ such that the diameter of every set in $\xi$ is less than $\delta$.
Then by SMB theorem, for $\mu$-a.e. $x\in X$,
$$\lim_{n\rightarrow \infty}-\frac{1}{|F_n|}\log \mu(\xi_{F_n}(x))=h_{\mu}(G,\xi)\le h_{\mu}(X,G).$$
Since $\xi_{F_n}(x)\subset B_{F_n}(x,\delta)$, we have that for $\mu$-a.e. $x\in X$,
$$\limsup_{n\rightarrow+\infty}-\frac{1}{|F_n|}\log \mu(B_{F_n}(x,\delta))\le h_{\mu}(X,G)$$
for ever $\delta$.

Now we show that
$$\lim\limits_{\delta\rightarrow 0}\liminf\limits _{n\rightarrow+\infty}-\frac{1}{|F_n|}\log \mu(B_{F_n}(x,\delta))\ge h_{\mu}(X,G).$$

For any $\epsilon>0$, let $\xi$ be a finite measurable partition of $X$ such that$$h_{\mu}(G,\xi)=\tilde{h}>h-\epsilon$$ and $\mu(\partial \xi)=0$,
where $\partial \xi$ is the union of the boundaries of atoms of $\xi$. Then $\forall \epsilon_1>0$, for
sufficiently small $\delta>0$, the $\delta-$neiborghhood of $\partial \xi$ (denoted by $U_{\delta}$) can have measure less than $\epsilon_1$.
By the ergodic theorem, $\frac{1}{|F_n|}\sum\limits_{g\in F_n}\chi_{U_{\delta}}(gx)$ converges to $\mu(U_{\delta})$ a.e..
Since
\begin{align*}
&\{x\in X: \lim_{n\rightarrow+\infty}\frac{1}{|F_n|}\sum\limits_{g\in F_n}\chi_{U_{\delta}}(gx)=\mu(U_{\delta})\}\\
\subset &\{x\in X: \lim_{n\rightarrow+\infty}\frac{1}{|F_n|}\sum\limits_{g\in F_n}\chi_{U_{\delta}}(gx)<\epsilon_1\}\\
=&\bigcup_{N\in \N}\{x\in X: \forall n>N, \frac{1}{|F_n|}\sum\limits_{g\in F_n}\chi_{U_{\delta}}(gx)<\epsilon_1\}\\
\triangleq & \bigcup_{N\in \N} E_N
\end{align*}
 and $E_N$ increases, for sufficiently large $N$, whence $n>N$,
\begin{align}\label{set-1}
 \mu(\{x\in X: \forall n'\ge n, \sum\limits_{g\in F_{n'}}\chi_{U_{\delta}}(gx)<\epsilon_1|F_{n'}|\})>1-\epsilon.
\end{align}

By the SMB Theorem \ref{th-3-2}, $-\frac{1}{|F_n|}\log \mu(\xi_{F_n}(x))$ converges to $\tilde{h}$ a.e.. Hence by the same argument as above,
for sufficiently large $N$, whence $n>N$,
\begin{align}\label{set-2}
 \mu(\{x\in X: \forall n'\ge n, -\frac{1}{|F_{n'}|}\log \mu(\xi_{F_{n'}}(x))>\tilde{h}-\epsilon\})>1-\epsilon.
\end{align}

Let \begin{align}\label{set-3}
      E=&\{x\in X: \forall n'\ge n, \sum\limits_{g\in F_{n'}}\chi_{U_{\delta}}(gx)<\epsilon_1|F_{n'}|\}\\
      &\bigcap\{x\in X: \forall n'\ge n,-\frac{1}{|F_{n'}|}\log \mu(\xi_{F_{n'}}(x))> \tilde{h}-\epsilon\}. \nonumber
    \end{align}
 Then for any $n>N$, $\mu(E)>1-2\epsilon$.

Let $w_{\xi,F_n}(x)=(\xi(gx))_{g\in F_{n}}$ be the $(\xi,F_n)-$name of $x$. For any $y\in B(x,\delta)$, we have that either $\xi(x)=\xi(y)$ or $x\in U_{\delta}(\xi)$.
Hence if $x\in E$ and $y\in B_{F_n}(x,\delta)$, then the Hamming distance between $w_{\xi,F_n}(x)$ and $w_{\xi,F_n}(y)$ is less than $\epsilon_1$. This implies that
when $x\in E$, $$B_{F_n}(x,\delta)\subset \bigcup \{\xi_{F_n}(y): w_{\xi,F_n}(y) \text{ is }\epsilon_1-\text{close to }w_{\xi,F_n}(x)\text{ under Hamming metric}\}.$$

By Stirling's formula, the total number of such $(\xi,F_n)-$names, denoted by $L_n$, can be estimated by:
\begin{align*}
 L_n\le\sum_{j=0}^{\lfloor\epsilon_1|F_n|\rfloor}C_{|F_n|}^j(\#\xi-1)^j\le \exp(K\epsilon|F_n|),
\end{align*}
where $K$ can be chosen as
$$K=\frac{1}{\epsilon}\big (\epsilon_1+\epsilon_1\log(\#\xi-1)-\epsilon_1\log\epsilon_1-(1-\epsilon_1)\log(1-\epsilon_1)\big )+2.$$
For the calculation of $K$, one may refer to \cite{K} or \cite{BK}.

We now note that $K>2$ is a constant only dependent on $\#\xi,\epsilon$ and $\epsilon_1$ but independent of $x$ and $n$. Moreover, $K\epsilon$ tends to $0$ while $\epsilon$ tends to $0$ (this can be done by making $\epsilon_1$ sufficiently small for fixed $\epsilon$).

Let $$D_n=\{x\in E: \mu(B_{F_n}(x,\delta))>\exp((-\tilde{h}+3K\epsilon)|F_n|)\}.$$

If we can prove that $\sum_{n=N}^{\infty}\mu(D_n)<\infty$, then apply the Borel-Cantelli Lemma:
for a.e. $x\in E$,
$$\liminf_{n\rightarrow+\infty}-\frac{1}{|F_n|}\log \mu(B_{F_n}(x,\delta))\ge \tilde{h}-3K\epsilon\ge h-(3K+1)\epsilon.$$
Hence we can obtain
$$\lim_{\delta\rightarrow 0}\liminf_{n\rightarrow+\infty}-\frac{1}{|F_n|}\log \mu(B_{F_n}(x,\delta))\ge h_{\mu}(X,G).$$

Now we estimate the measure of $D_n$.

For any $x\in D_n$, in those $L_n-$many $(\xi,F_n)-$names which are $\epsilon_1-$close to $w_{\xi,F_n}(x)$ in Hamming distance,
there exists at least one corresponding atom of $\xi_{F_n}$ whose measure is greater than $\exp((-\tilde{h}+2K\epsilon)|F_n|)$.
The total number of such atoms will not exceed
$\exp(\tilde{h}-2K\epsilon|F_n|)$. Hence $Q_n$, the total number of elements in $\xi_{F_n}$ that intersect $D_n$, satisfies:
$$Q_n\le L_n\exp((\tilde{h}-2K\epsilon)|F_n|)\le \exp((\tilde{h}-K\epsilon)|F_n|).$$
Let $S_n$ denote the total measure of such $Q_n$ elements of $\xi_{F_n}$ whose intersections with $E$ have positive measure. Then from \eqref{set-3},
$$S_n\le Q_n\exp((-\tilde{h}+\epsilon)|F_n|)\le\exp((-K\epsilon+\epsilon)|F_n|),$$
which follows that
$$\mu(D_n)\le S_n\le \exp((-K\epsilon+\epsilon)|F_n|).$$

From the increasing condition \eqref{eq-1-2}, for sufficiently large $N$, whenever $n\ge N$, $\frac{|F_n|}{\log n}\ge \frac{2}{(K-1)\epsilon}$ holds.
Then $\exp((-K\epsilon+\epsilon)|F_n|)\le n^{-2}$ and hence $\sum_{n=N}^{\infty}\mu(D_n)\le\sum_{n=N}^{\infty}S_n<\infty$.
Thus the proof is completed.
\end{proof}

\section{A variational principle for Bowen entropy}
In this section we will show that there exists a variational principle between Bowen topological entropy and (lower) local entropy.
\begin{theorem}\label{th-4-1}
Let $(X,G)$ be a compact metric $G-$action topological dynamical system and $G$ a discrete countable amenable group. If $K\subseteq X$ is non-empty and compact
and $\{F_n\}$ a sequence of finite subsets in $G$ with the increasing condition $\lim\limits_{n\rightarrow+\infty}\frac{|F_n|}{\log n}=\infty$,
then
$$h^B_{top}(K, \{F_n\})=\sup\{\underline h_{\mu}^{loc}(\{F_n\}):\mu(K)=1\},$$
where the supremum is taken over $\mu\in M(X)$, the Borel probability measures on $X$.
\end{theorem}
We remark here that in this theorem, we do not need $\{F_n\}$ to be a F{\o}lner sequence.

For the proof of this theorem, we use the method from \cite{FH}, which extends the ideas and techniques from geometric measure theory (cf. \cite{F,M}) to dynamical system. 

We first introduce the so-called weighted entropy for system $(X,G)$ in the following way.
Let $\{F_n\}$ be a sequence of finite subsets in $G$ with $|F_n|$ tends to infinity.
For any function $f : X\rightarrow[0,+\infty), N \in\mathbf{N}$
and $\epsilon> 0$, define
$$\mathcal{W}(f,N,\epsilon,s,\{F_n\})=\inf \sum_{i}c_i\exp(-s|F_{n_i}|),$$
where the infimum is taken over all finite or countable families $\{(B_{F_{n_i}}(x_i,\epsilon), c_i)\}$ such
that $0 < c_i <+\infty, x_i\in X, n_i\ge N$ and $\sum_{i}c_i\chi_{B_{F_{n_i}}(x_i,\epsilon)}\ge f.$

For $Z\subset X$ and $f =\chi_Z$ we set $\mathcal{W}(Z,N,\epsilon,s,\{F_n\})=\mathcal{W}(\chi_Z,N,\epsilon,s,\{F_n\})$.
The quantity $\mathcal{W}(Z,N,\epsilon,s,\{F_n\})$ does
not decrease as $N$ increases and $\epsilon$ decreases, hence the following limits exist:
$$\mathcal{W}(Z,\epsilon,s,\{F_n\})=\lim_{N\rightarrow +\infty}\mathcal{W}(Z,N,\epsilon,s,\{F_n\}),
\mathcal{W}(Z,s,\{F_n\})=\lim_{\epsilon\rightarrow 0}\mathcal{W}(Z,\epsilon,s,\{F_n\}).$$
Clearly, there exists a critical value of the parameter $s$, which we will denote by $h^{WB}_{top}(Z,\{F_n\})$, where $\mathcal{W}(Z,s,\{F_n\})$ jumps
from $+\infty$ to $0$, i.e.
\begin{equation*}
\mathcal{W}(Z,s,\{F_n\})=\begin{cases}
0, s > h^{WB}_{top}(Z,\{F_n\}),\\
+\infty, s < h^{WB}_{top}(Z,\{F_n\}).
\end{cases}
\end{equation*}
We call $h^{WB}_{top}(Z,\{F_n\})$ the weighted Bowen topological entropy along $\{F_n\}$ restricted to Z or,
simply, the weighted Bowen topological entropy of Z along $\{F_n\}$.

Now we will consider the relation between the Bowen topological entropy and the weighted Bowen topological entropy.
It is clear that if we take $f=\chi_Z$ and $c_i=1$, then the following holds.
\begin{proposition}\label{prop-4-2}
$\mathcal{W}(Z,N,\epsilon,s,\{F_n\})\le \mathcal{M}(Z,N,\epsilon,s,\{F_n\})$, for any $s,\epsilon>0$ and $N\in\mathbf{N}$.
\end{proposition}

By alternating some parameters, we can get
\begin{proposition}\label{prop-4-3}
$\mathcal{M}(Z,N,6\epsilon,s+\delta,\{F_n\})\le \mathcal{W}(Z,N,\epsilon,s,\{F_n\})$, for any $s,\epsilon,\delta>0$ and sufficient large $N\in\mathbf{N}$.
\end{proposition}
The proof is similar to Proposition 3.2 of \cite{FH}.
But we should remark here that condition \eqref{eq-1-2} ensures the existence of the parameter $N$.

\begin{proof}[Proof of Proposition \ref{prop-4-3}.]

From the increasing condition \eqref{eq-1-2}, there exists $N>2$, such that for $n\ge N$, $\frac{|F_n|}{\ln n}\ge 2\delta^{-1}$.
Then $e^{-\delta |F_n|}\le n^{-2}$ and hence $\sum_{n=N}^{\infty}e^{-\delta |F_n|}<1$.

Let $\{(B_{F_{n_i}}(x_i,\epsilon), c_i)\}_{i\in \mathcal{I}}$ be a countable family such that $\mathcal{I}\subset \mathbf{N}, x_i\in X, 0 < c_i <+\infty, n_i\ge N$ and
$$\sum_{i\in\mathcal{I}}c_i\chi_{B_i}\ge \chi_Z,$$ where $B_i := B_{F_{n_i}}(x_i, \epsilon)$. Then we will show that
$$\mathcal{M}(Z,N,6\epsilon,s+\delta,\{F_n\})\le \sum_{i\in\mathcal{I}}c_i\exp(-s|F_{n_i}|),$$
and hence $$\mathcal{M}(Z,N,6\epsilon,s+\delta,\{F_n\})\le \mathcal{W}(Z,N,\epsilon,s,\{F_n\}).$$

Decompose $\mathcal{I}$ into subsets $\mathcal{I}_n := \{i\in\mathcal{I} : n_i = n\}$ and
the finite subsets $\mathcal{I}_{n,k} = \{i\in\mathcal{I}_n : i\le k\}$ for $n\ge N$ and $k\in\mathbf{N}$.
Write for brevity $B_i := B_{F_{n_i}}(x_i,\epsilon)$ and $5B_i := B_{F_{n_i}}(x_i,5\epsilon)$ for $i\in\mathcal{I}$. We
may assume $B_i$'s are mutually different. For $t > 0$, set
$$Z_{n,t} =\{x\in Z : \sum_{i\in\mathcal{I}_n}c_i\chi_{B_i}(x)>t\}$$
and $$Z_{n,k,t} =\{x\in Z : \sum_{i\in\mathcal{I}_{n,k}}c_i\chi_{B_i}(x)>t\}.$$

For $Z_{n,k,t}$, we may assume that each $c_i$ is a positive integer. Since $\mathcal{I}_{n,k}$ is finite and by approximating the $c_i$'s from above, we may first assume
$c_i$'s are positive rational numbers. Also notice that $Z_{n,k,dt}$ for $dc_i$'s is equal to $Z_{n,k,t}$ for $c_i$'s, so multiplying with a common denominator $d$, we may
assume that each $c_i$ is a positive integer. Let $m$ be the least integer with $m\ge t$.
Denote $\mathcal{B} = \{B_i, i\in\mathcal{I}_{n,k}\}$ and define $u : \mathcal{B}\rightarrow Z$ by $u(B_i) = c_i$. We define by
induction integer-valued functions $v_0, v_1, \cdots , v_m$ on $\mathcal{B}$ and sub-families $\mathcal{B}_1, \cdots, \mathcal{B}_m$ of
$\mathcal{B}$ starting with $v_0 = u$. Using the classical 5r-coving Lemma in fractal geometry (see, for example, Theorem 2.1 of \cite{M}), taking the metric $d_{F_n}$ instead
of $d$, there exists a pairwise disjoint subfamily $\mathcal{B}_1$ of $\mathcal{B}$ such that $\bigcup_{B\in\mathcal{B}}B\subset \bigcup_{B\in\mathcal{B}_1}5B$,
and hence $Z_{n,k,t}\subset \bigcup_{B\in\mathcal{B}_1} 5B$. Repeating this process, we can define
inductively for $j = 1, \cdots ,m$, disjoint subfamilies $\mathcal{B}_j$ of $\mathcal{B}$ such that
$$\mathcal{B}_j\subset\{B\in\mathcal{B} : v_{j-1}(B)\ge 1\}, Z_{n,k,t}\subset \bigcup_{B\in\mathcal{B}_j}5B$$
and the functions $v_j$ such that
\begin{equation*}
v_j(B) =\begin{cases}v_{j-1}(B)-1 \text{ for } B\in\mathcal{B}_j ,\\
                     v_{j-1}(B) \text{ for } B\in\mathcal{B}\setminus \mathcal{B}_j .
\end{cases}
\end{equation*}
Since for $j < m, Z_{n,k,t}\subset\{x :\sum_{B\in\mathcal{B}: B\ni x} v_j(B)\ge m-j\}$, whenever
every $x\in Z_{n,k,t}$ belongs to some ball $B\in\mathcal{B}$ with $v_j(B)\ge1$, the above inductive process works. Thus
\begin{align*}
\sum_{j=1}^m\#(\mathcal{B}_j)e^{-s|F_n|} &=\sum_{j=1}^m\sum_{B\in\mathcal{B}_j}(v_{j-1}(B)-v_j(B))e^{-s|F_n|}\\
&\le \sum_{B\in\mathcal{B}}\sum_{j=1}^m(v_{j-1}(B)-v_j(B))e^{-s|F_n|}\\
&\le\sum_{B\in\mathcal{B}}u(B)e^{-s|F_n|}=\sum_{i\in\mathcal{I}_{n,k}}c_ie^{-s|F_n|}.
\end{align*}
Choose $j_0\in\{1, \cdots ,m\}$ such that $\#(\mathcal{B}_{j_0})$ is the smallest. Then
$$\#(\mathcal{B}_{j_0})e^{-s|F_n|}\le \frac{1}{m}\sum_{i\in\mathcal{I}_{n,k}}c_ie^{-s|F_n|}\le\frac{1}{t}\sum_{i\in\mathcal{I}_{n,k}}c_ie^{-s|F_n|}.$$

This shows that: for each $n\ge N, k\in\mathbf{N}$ and $t > 0$, there exists a finite set $\mathcal{J}_{n,k,t}\subset \mathcal{I}_{n,k}$
such that the balls $B_i$ ($i\in \mathcal{J}_{n,k,t}$) are pairwise disjoint, $Z_{n,k,t}\subset \bigcup_{i\in\mathcal{J}_{n,k,t}}5B_i$
and $$\#(\mathcal{J}_{n,k,t})e^{-s|F_n|}\le\frac{1}{t}\sum_{i\in\mathcal{I}_{n,k}}c_ie^{-s|F_n|}.$$

Now assume $Z_{n,t}\neq\emptyset$. Since $Z_{n,k,t}\uparrow Z_{n,t}$,
$Z_{n,k,t}\neq\emptyset$ when $k$ is large enough. Let $\mathcal{J}_{n,k,t}$ be the sets constructed above. Then
$\mathcal{J}_{n,k,t}\neq\emptyset$ when $k$ is large enough. Define $E_{n,k,t} = \{x_i : i\in\mathcal{J}_{n,k,t}\}$. Note that
the family of all non-empty compact subsets of $X$ is compact with respect to the
Hausdorff distance. It follows that there is a subsequence
$\{k_j\}$ of natural numbers and a non-empty compact set $E_{n,t}\subset X$ such that $E_{n,k_j ,t}$
converges to $E_{n,t}$ in the Hausdorff distance as $j\rightarrow+\infty$. Since any two points in $E_{n,k,t}$
have a distance (with respect to $d_{F_n}$) no less than $\epsilon$, so do the points in $E_{n,t}$. Thus
$E_{n,t}$ is a finite set, moreover, $\#(E_{n,k_j ,t}) = \#(E_{n,t})$ when $j$ is large enough. Hence
$$\bigcup_{x\in E_{n,t}}B_{F_n}(x, 5.5\epsilon)\supset\bigcup_{x\in E_{n,k_j,t}}B_{F_n}(x, 5\epsilon)=\bigcup_{i\in \mathcal{J}_{n,k_j,t}}5B_i\supset Z_{n,k_j,t}$$
when j is sufficiently large, and thus
$\bigcup_{x\in E_{n,t}}B_{F_n}(x, 6\epsilon)\supset Z_{n,t}$. Since
$\#(E_{n,k_j ,t}) = \#(E_{n,t})$ when $j$ is large enough, we have $\#(E_{n,t})e^{-s|F_n|}\le\frac{1}{t}\sum_{i\in\mathcal{I}_{n}}c_ie^{-s|F_n|}$.
This forces
\begin{align*}
\mathcal{M}(Z_{n,t},N,6\epsilon,s+\delta,\{F_n\})&\le \#(E_{n,t})e^{-(s+\delta)|F_n|}\\ &\le\frac{1}{e^{\delta |F_n|}t}\sum_{i\in\mathcal{I}_{n}}c_ie^{-s|F_n|}.
\end{align*}

Since $\sum_{n=N}^{\infty}e^{-\delta |F_n|}<1$, we can deduce that
$Z\subset\bigcup_{n=N}^{\infty} Z_{n,e^{-\delta |F_n|}t}$ for any $t\in (0,1)$. And also
note that $\mathcal{M}(Z,N,\epsilon,s,\{F_n\})$ is an outer measure of $X$, we have
\begin{align*}\mathcal{M}(Z,N,6\epsilon,s+\delta,\{F_n\})&\le\sum_{n=N}^{\infty}\mathcal{M}(Z_{n,e^{-\delta |F_n|}t},N,6\epsilon,s+\delta,\{F_n\})\\
&\le\sum_{n=N}^{\infty}\frac{1}{t}\sum_{i\in\mathcal{I}_{n}}c_ie^{-s|F_n|}\\ &=\frac{1}{t}\sum_{i\in\mathcal{I}}c_ie^{-s|F_{n_i}|}.
\end{align*}
Hence
$$\mathcal{M}(Z,N,6\epsilon,s+\delta,\{F_n\})\le \sum_{i\in\mathcal{I}}c_ie^{-s|F_{n_i}|}.$$
This finishes the proof of the proposition.
\end{proof}

The following is a dynamical Frostman's lemma related to the weighted Bowen entropy in the amenable group action case.

\begin{lemma}\label{lem-4-4}
Let $K$ be a non-empty compact subset of $X$. Let $s\ge0, N\in\mathbf{N}$ and
$\epsilon> 0$. Suppose that $c :=\mathcal{W}(K,N,\epsilon,s,\{F_n\})> 0$. Then there is a Borel probability measure $\mu$
on $X$ such that $\mu(K) = 1$ and
$$\mu(B_{F_n}(x,\epsilon))\le\frac{1}{c}e^{-s|F_n|}, \forall x\in X, n\ge N.$$
\end{lemma}
\begin{proof}
For the proof one may refer to Lemma 3.4 of \cite{FH} and we omit it here.
\end{proof}

\begin{proof}[\bf Proof of Theorem \ref{th-4-1}]
For any $\mu\in M(X)$ with $\mu(K) = 1$, $x\in X,n\in\mathbf{N}$ and $\epsilon>0$, recall that
$$\underline h_{\mu}^{loc}(x,\epsilon,\{F_n\})=\liminf_{n\rightarrow +\infty}-\frac{1}{|F_n|}\log \mu(B_{F_n}(x,\epsilon)).$$
Since $\underline h_{\mu}^{loc}(x,\epsilon,\{F_n\})$ is nonnegative and increases as $\epsilon$ decreases,
by the monotone convergence theorem,
$$\lim_{\epsilon\rightarrow 0}\int_X \underline h_{\mu}^{loc}(x,\epsilon,\{F_n\})d\mu=\int_X\underline h_{\mu}^{loc}(x,\{F_n\})d\mu=\underline h_{\mu}^{loc}(\{F_n\}).$$

Fix $\epsilon > 0$ and $\ell \in\mathbf{N}$. Denote $u_{\ell} = \min\{\ell,\int_X \underline h_{\mu}^{loc}(x,\epsilon,\{F_n\})d\mu-\frac{1}{\ell}\}$.
Then there exist a Borel set $A_{\ell}\subset X$ with $\mu(A_{\ell}) > 0$ and $N\in\mathbf{N}$ such that
$$\mu(B_{F_n}(x,\epsilon))\le e^{-u_{\ell}|F_n|}, \forall x\in A_{\ell}, n\ge \mathbf{N}.$$
Now let $\{B_{F_{n_i}}(x_i, \epsilon/2)\}$ be a countable or finite family such that $x_i\in X, n_i\ge N$ and
$\bigcup_i B_{F_{n_i}}(x_i,\epsilon/2)\supset K\cap A_{\ell}$. We may assume that for each $i, B_{F_{n_i}}(x_i, \epsilon/2)\cap K\cap A_{\ell}\neq\emptyset$,
and choose $y_i\in B_{F_{n_i}}(x_i, \epsilon/2)\cap K\cap A_{\ell}$. Then we have
\begin{align*}
\sum_{i}e^{-u_{\ell}|F_{n_i}|}&\ge \sum_{i}\mu(B_{F_{n_i}}(y_i,\epsilon))\ge \sum_{i}\mu(B_{F_{n_i}}(x_i,\epsilon/2))
\ge\mu(K\cap A_{\ell})=\mu(A_{\ell}) > 0
\end{align*}
and it follows that $$\mathcal{M}(K,u_{\ell},\{F_n\})\ge\mathcal{M}(K,u_{\ell},N,\epsilon/2,\{F_n\})\ge\mathcal{M}(K\cap A_{\ell},u_{\ell},N,\epsilon/2,\{F_n\})\ge\mu(A_{\ell}).$$
Therefore $h^B_{top}(K,\{F_n\})\ge u_{\ell}$.

Letting $\ell\rightarrow+\infty$, we have the inequality
$$h^B_{top}(K,\{F_n\})\ge\int_X\underline h_{\mu}^{loc}(x,\epsilon,\{F_n\})d\mu.$$
Hence $h^B_{top}(K,\{F_n\})\ge\underline h_{\mu}^{loc}(\{F_n\})$.

Now we will show that $$h^B_{top}(K,\{F_n\})\le\sup\{\underline h_{\mu}^{loc}(\{F_n\}):\mu\in M(X),\mu(K)=1\}.$$
We can assume that $h^B_{top}(K,\{F_n\}) > 0$. By Proposition \ref{prop-4-2} and \ref{prop-4-3}, $h^{WB}_{top}(K,\{F_n\})= h^B_{top}(K,\{F_n\})$.
Let $0 < s < h^B_{top}(K,\{F_n\})$. Then there exist $\epsilon> 0$ and $N\in\mathbf{N}$ such that $c := \mathcal{W}(K,N,\epsilon,s,\{F_n\}) > 1$.
By Lemma \ref{lem-4-4}, there exists $\mu\in M(X)$ with $\mu(K) = 1$ such that $\mu(B_{F_n}(x,\epsilon))\le\frac{1}{c}e^{-s|F_n|}$
for any $x\in X$ and $n\ge N$. Clearly
$\underline h_{\mu}^{loc}(x,\{F_n\})\ge\underline h_{\mu}^{loc}(x,\epsilon,\{F_n\})\ge s$ for each $x\in X$
and hence $\underline h_{\mu}^{loc}(\{F_n\})=\int_X \underline h_{\mu}^{loc}(x,\{F_n\})d\mu(x)\ge s$.
This finishes the proof of the theorem.
\end{proof}

\section{Proof of the main resultes}
Now we give the proof of Theorem \ref{th-1-1}.

We first prove that for any F{\o}lner sequence $\{F_n\}$ and any open cover $\mathcal {U}$,
$$h_{top}^B(\mathcal {U}, X, \{F_n\})\le h_{top}(G,\mathcal{U}).$$

Let $\mathcal {V}$ be a subcover of $\mathcal {U}_{F_n}$ with minimal cardinality. We can write $\mathcal {V}=\{X(\mathbf{U}): \mathbf{U}\in \Lambda\}$, where $\Lambda\subset W_{F_n}(\mathcal{U})$. Then the cardinality of $\Lambda$ equals to $N(\mathcal {U}_{F_n})$. Hence
$$\sum\limits_{\mathbf{U}\in \Lambda}\exp(-s m(\mathbf{U}))=N(\mathcal {U}_{F_n})e^{-s|F_n|},$$
which implies that
$$\mathcal{M}(X,\mathcal{U},s,\{F_n\})\le \exp\big((-s+\frac{1}{|F_n|}\log N(\mathcal {U}_{F_n}))|F_n|\big).$$
If $s$ is larger than $h_{top}(G,\mathcal{U})$, $\mathcal{M}(X,\mathcal{U},s,\{F_n\})=0$. So
$h_{top}^B(\mathcal {U}, X, \{F_n\})\le h_{top}(G,\mathcal{U})$.

In the following we show that $h_{top}^B(X, \{F_n\})\ge h_{top}(X, G)$ for tempered F{\o}lner sequence $\{F_n\}$ with
$\lim\limits_{n\rightarrow+\infty}\frac{|F_n|}{\log n}=\infty$. For the proof we need the following classical variational principle for amenable group action dynamical systems, see \cite{OP,ST}.
\begin{theorem}[Variational principle for topological entropy]\label{th-5-1}
$$
  h_{top}(X, G)=\sup\limits_{\mu\in M(X,G)}h_{\mu}(X,G)=\sup\limits_{\mu\in E(X,G)}h_{\mu}(X,G),
$$
where $M(X,G)$ and $E(X,G)$ are the collection of $G-$invariant and $G-$ergodic Borel probability measures of $X$ respectively.
\end{theorem}

We note that all the entropies in the above theorem do not dependent on the choice of F{\o}lner sequence.

Then by Theorem \ref{th-3-3}, \ref{th-4-1} and \ref{th-5-1},
\begin{align*}
    h_{top}(X,G)&=h_{top}(X,G,\{F_n\})\\
    &=\sup\{h_{\mu}(X,G,\{F_n\}):\mu\in E(X,G)\}\\
    &=\sup\{\underline h_{\mu}^{loc}(\{F_n\}):\mu\in E(X,G)\}\\
    &\le\sup\{\underline h_{\mu}^{loc}(\{F_n\}):\mu\in M(X)\}\\
    &= h_{top}^B(X,\{F_n\}),
\end{align*}
where $h_{top}(X,G,\{F_n\})$ and $h_{\mu}(X,G,\{F_n\})$ denote the topological entropy and measure theoretic entropy of $(X,G)$ for the F{\o}lner sequence
$\{F_n\}$ respectively. This completes the proof of Theorem \ref{th-1-1}.

To end this paper, we ask the following question:
\begin{question}
Can Theorem \ref{th-1-1} be proved through a pure topological way?
\end{question}

{\bf Acknowledgements}
The research was supported by the National Basic Research Program of China (Grant No. 2013CB834100) and the National Natural
Science Foundation of China (Grant No. 11271191, 11431012).


\begin{thebibliography}{99}
\bibitem{B} R. Bowen, Topological entropy for noncompact sets, Trans. Amer. Math. Soc. 184 (1973) 125--136.
\bibitem{BK} M. Brin, A. Katok, On Local Entropy, Lecture Notes in Mathematics, vol.1007, Springer, Berlin, 1983, 30--38.
\bibitem{F} H. Federer, Geometric Measure Theory, Springer-Verlag, New York, 1969.
\bibitem{FH} D.J. Feng, W. Huang, Variational principles for topological entropies of subsets, J. Funct. Anal. 263(8) (2012) 2228--2254.
\bibitem{L} E. Lindenstrauss, Pointwise theorems for amenable groups, Invent. Math. 146 (2001) 259--295.
\bibitem{K} A. Katok, Lyapunov exponents, entropy and periodic orbits for diffeomorphisms, Publ.Math. I.H.E.S., v.51(1980), 137--173.
\bibitem{M} P. Mattila, Geometry of Sets and Measures in Euclidean Spaces, Cambridge University Press, 1995.
\bibitem{OP} J.M. Ollagnier, D. Pinchon, The variational principle, Studia Math. 72 (2) (1982) 151--159.
\bibitem{OW} D.S. Ornstein, B. Weiss, Entropy and isomorphism theorems for actions of amenable groups, J. Anal. Math. 48(1987) 1--141.
\bibitem{P} Ya. B. Pesin, Dimension Theory in Dynamical Systems, Contemporary Views and Applications, University of Chicago Press, Chicago, IL, 1997.
\bibitem{S} A. Shulman, Maximal ergodic theorems on groups, Dep. Lit. NIINTI, No.2184, 1988.
\bibitem{ST} A.M. Stepin, A.T. Tagi-Zade, Variational characterization of topological pressure of the amenable groups of transformations,
Dokl. Akad. Nauk SSSR 254 (3) (1980) 545--549 (in Russian).

\bibitem{W} B. Weiss, Actions of amenable groups, Topics in Dynamics and Ergodic Theory. (2003) 226--262. London Math. Soc. Lecture Note Ser., 310, Cambridge Univ. Press,
Cambridge, 2003.

\end{thebibliography}
\end{document}